\documentclass[11pt]{article}

\usepackage{amsfonts}
\usepackage{mathtools}
\usepackage{amsmath}
\usepackage{amssymb}
\usepackage{amsthm}
\usepackage{enumerate}

\def \Z {\mathbb Z}
\def \E {\mathbb E}

\def \R {\mathbb R}

\def \C {\mathbb C}
\def\P{\mathbb P}

\def\ep{\varepsilon}

\def\Om{\Omega}

\def\Ph1{P^{(h_1)}}
\def\Ph2{P^{(h_2)}}

\def\b0{{\bf 0}}

\newtheorem{thm}{Theorem}[section]

\newtheorem{lem}[thm]{Lemma}
\newtheorem{cor}[thm]{Corollary}
\newtheorem{obs}[thm]{Observation}

\theoremstyle{plain}

\newtheorem{claim}[thm]{Claim}
\newtheorem{rem}[thm]{Remark}

\theoremstyle{definition}

\begin{document}

\title{An OSSS-type inequality for uniformly drawn subsets of fixed size}

\author{Jacob van den Berg\footnotemark[2], Henk Don\footnotemark[3] 
\small \\
  {\small \footnotemark[2] CWI, Amsterdam ; email: J.van.den.Berg@cwi.nl} \\
  {\small \footnotemark[3] Radboud Universiteit, Nijmegen; email: henk.don@ru.nl}
  }
\date{}
\maketitle

\begin{abstract}
 The OSSS inequality (O'Donnell, Saks, Schramm and Servedio \cite{OSSS05}) gives an upper bound for the variance of a function $f$ of
independent $0$-$1$ valued random variables in terms of the influences of these random variables and the computational complexity of a (randomised) 
algorithm for determining the value of $f$.

Duminil-Copin, Raoufi and Tassion \cite{DRT19} obtained a generalization to monotonic measures 
and used it to prove 
new results for Potts models and random-cluster models.
Their generalization of the OSSS inequality raises the question if there are still other measures for which
a version of that inequality holds.

We derive a version of the OSSS inequality for a family of measures that are far from monotonic,
namely the $k$-out-of-$n$ measures (these measures correspond with drawing $k$ elements from a set of size $n$ uniformly). 
We illustrate the inequality by studying the
event that there is an occupied horizontal crossing of an $R \times R$ box on the triangular lattice
in the site percolation model where exactly half of the
vertices in the box are occupied.

\end{abstract}

\begin{section}{Introduction}\label{Intr}

\begin{subsection}{Background and main results of the paper}\label{backgr}
The OSSS inequality \cite{OSSS05} gives an upper bound for the variance of a function $f$ of independent $0$-$1$ valued
random variables in terms of their influences on $f$ and their revealment probabilities with respect to an algorithm to determine the value of $f$.
Since the break-through work by Duminil-Copin, Raoufi and Tassion around 2017 it has become one of the main tools to prove sharp phase transition in
a number of important models from statistical mechanics. In particular, in \cite{DRT19a} they applied the OSSS inequality to Voronoi percolation on $\R^d$,
and in \cite{DRT19} they extended the inequality to  monotonic measures (measures satisfying the FKG lattice condition) and used that extension to prove sharp phase
transition for random-cluster and Potts models. 
That version of OSSS was further extended, but still under the condition that the measure is monotonic, by Hutchcroft \cite{Hu20} (who developed a two-function form
and used it to obtain new critical exponent inequalities) and by Dereudre and Houdebert \cite{DH21}
(who gave a generalization in a continuum setting and used it to prove sharp phase
transition for the Widom-Rowlinson model).

The above raises the natural question whether the OSSS inequality can be extended beyond the class of monotonic measures.
Our main result, Theorem \ref{Maint} below, is a version of the OSSS inequality for a family of measures which are clearly non-monotonic, namely $k$-out-of-$n$ measures.

\smallskip
Before stating the theorem, we will introduce some definitions and notation.

If $s \geq 1$ is an integer, $[s]$ denotes the set $\{1, \cdots, s\}$. If $t \geq s$ is also an integer, $[s,t]$ denotes $\{s, \cdots, t\}$.

Let $E$ be a finite set and let $k \leq |E|$, where $|E|$ denotes the size (i.e. cardinality) of $|E|$. If $\omega \in \{0,1\}^E$, we will use
the notation $|\omega|$ for $\sum_{e \in E} \omega_e$. Informally, the $k$-out-of-$E$ distribution (notation $P_{k,E}$) is the uniform distribution 
on the set of all subsets of $E$ of size $k$. Formally, we define
$P_{k,E}$ as the following distribution on $\{0,1\}^E$: For each $\omega \in \{0,1\}^E$,

\begin{equation}\label{PkE-def}
P_{k,E}(\omega) =
\begin{cases}
0, & \text{ if }|\omega| \neq k \\
\frac{1}{\binom{|E|}{k}}, & \text{ if }|\omega| = k,
\end{cases}
\end{equation}
Often, when the
set $E$ is clear from the context, or only its size $|E|$ matters, we will simply write $P_{k,n}$ instead of $P_{k,E}$, with $n= |E|$
(and call it a $k$-out-of-$n$ distribution).

Let $\omega \in \{0,1\}^E$ and and let $A$ be an event, i.e. a subset of $\{0,1\}^E$.
For convenience we will assume that $A \neq \emptyset$ and $A \neq  \{0,1\}^E$.
An element $e \in E$ is said to be {\it pivotal} (w.r.t. $\omega$
and $A$) if exactly one of $\omega$ and $\omega^{(e)}$ is in $A$. Here $\omega^{(e)}$ denotes the element of $\{0,1\}^E$ obtained
from $\omega$ by replacing $\omega_e$ by $1-\omega_e$.
We say that $e$ is $0$-pivotal if $e$ is pivotal and $\omega_e = 0$.
Further, a {\it pair} of points $e, f \in E$ with $e \neq f$ is called pivotal
(w.r.t. $\omega$
and $A$) if exactly one of $\omega$ and $\omega^{(e,f)}$ is in $A$. Here $\omega^{(e,f)}$ denotes the element of $\{0,1\}^E$ obtained
from $\omega$ by exchanging the values of $\omega_e$ and $\omega_f$. 

The event $A$ is called {\it increasing} if, for all pairs $\omega, \sigma$ with $\omega \in A$ and
$\sigma \in \{0,1\}^E$, $\omega \leq \sigma$ implies that $\sigma \in A$. Here, as usual, $\omega \leq \sigma$ means that
$\omega_e \leq \sigma_e$ for all $e \in E$.
In this paper we will mainly deal with increasing events, and it will turn out that in particular the above mentioned notion of being
$0$-pivotal comes up naturally. We will use the notation $I_{k, E}^A(e)$ for the probability (under the measure $P_{k,E}$) that a given point $e$ is $0$-pivotal
w.r.t. the event $A$.
Often, when $k$, $E$ and $A$ are clear from the context, we simply write $I(e)$.


Somewhat informally, a {\it decision tree} is (in the present context) an algorithm to check whether an `input string' $\omega$ is in $A$, by `examining' 
one by one the elements
of $E$ (i.e. querying their $\omega$-value $\omega_e$), where at each step the next element of $E$ to be examined depends on the already examined elements of $E$
and their $\omega$-values. Note that if, after some step, the up to then revealed $\omega$-values already determine whether the
input string is in $A$ or not, the algorithm may stop. That step wil be denoted by $\tau$. 

 We will use the {\it formal} definition of a decision tree from \cite{DRT19}: A decision tree is a pair 
$T = (e_1, \phi)$, with $e_1 \in E$ and $\phi$ of the form $(\phi_2, \cdots, \phi_n)$, where each $\phi_t$ is a function which
assigns, to each pair $\left( (e_1, \cdots, e_{t-1}), \,\, \omega_{(e_1, \cdots, e_{t-1})} \right)$ an element $e_t$ of $E \setminus \{e_1, \cdots, e_{t-1}\}$.

The corresponding algorithm is then as follows: First $e_1$ is examined, i.e. its value $\omega_{e_1}$ is queried (and revealed). Depending
on that value the next element of $E$, namely the element $e_2 := \phi_2(e_1, \omega_{e_1})$ is selected and examined. After the value $\omega_{e_2}$
of $e_2$ has been revealed, the next element, namely $e_3 := \phi_3((e_1, e_2), (\omega_{e_1}, \omega_{e_2})$ is examined etcetera.
The time $\tau$ mentioned before is then formally defined as 
\begin{equation}\label{def-tau}
\tau(\omega) := \min\{t \geq 1 \, : \, \text{ for all } \omega' \in \{0,1\}^E \text{ with }
\omega'_{e_{[t]}} = \omega_{e_{[t]}}, \,\, I_A(\omega') = I_A(\omega)\},
\end{equation}
where the notation $e_{[t]}$ is used for $\{e_1, \cdots, e_t\}$.
Note that $\tau$ depends on $A$ and $T$, although this is not visible in the notation.

Often it is assumed, and this is also the case in this paper, that the string $\omega$ is generated according to a probability distribution $\P$.
One could say that a given decision tree is
`efficient' (w.r.t. the above mentioned $\P$) if `typically', or `on average', the number $\tau$ is `small'.
In this paper we focus on the case where $\P = P_{k,E}$. With this $\P$ it makes sense to replace, in the definition \eqref{def-tau} of $\tau$,
the part (for all $\omega'  \in \{0,1\}^E $) by (for all $\omega'  \in \{0,1\}^E$ with $|\omega'| = k$). Although the corresponding
(slightly stronger) analog of our Theorem \ref{Maint} would still be true (with practically the same proof) we keep the more `standard' definition \eqref{def-tau}.

The probability that the value at a given `site' $e \in E$ is revealed (for the decision tree $T$) will be denoted
by $\delta_e = \delta_e(A, T)$. So, formally,
$$\delta_e(A, T) = P_{k,E}(e_t = e \text{ for some } t \leq \tau(\omega)).$$
 
Our main result, Theorem \ref{Maint} below, is a version of the OSSS inequality for $k$-out-of-$n$ measures.
\begin{thm}\label{Maint} 
(a) There is a constant $C$ such that, for every set $E$ with even cardinality, every increasing event $A \subset \{0,1\}^E$,
and every decision tree $T$,

\begin{eqnarray}\label{eq-maint-a}
&\, & \,\,\,\, \, \, \, \,\,P_{\frac{|E|}{2},E}(A) (1 - P_{\frac{|E|}{2},E}(A)) \\ \nonumber
&\, & \leq \,  C\, \left( \sum_{e \in E} I(e) \, \delta_e + \sum_{e \in E} I(e) \bar{\delta}\right ),
\end{eqnarray}
where $ \bar{\delta} = \bar{\delta}(A,T) = \frac{1}{|E|} \sum_{e \in E} \delta_e(A, T)$, the `average revealment per vertex'.

(b) More generally, for every $\ep \in (0, 1/2]$ there is a constant $C(\ep)$ such that for all positive integers $n$, 
all $E$ with $|E| = n$, all integers 
$k \in [\ep n, (1-\ep)n]$, all increasing events $A \subset \{0,1\}^E$ and all decision trees $T$,

\begin{eqnarray}\label{eq-maint-b}
&\, & \,\,\,\, \, \, \, \,\,P_{k,E}(A) (1 - P_{k,E}(A)) \\ \nonumber
&\, & \leq \,  C(\ep) \, \left( \sum_{e \in E} I(e) \, \delta_e + \sum_{e \in E} I(e) \bar{\delta}\right ).
\end{eqnarray}

{\small Although we haven't pursued this extensively, more quantitative versions can be obtained from our proof. For instance, if $|E| \geq 10$,
then \eqref{eq-maint-a} holds for $C = 20$.} 
\end{thm}

%

%


\begin{rem}\label{rem-maint}

As mentioned before, the OSSS inequality in \cite{DRT19} gives an upper bound for the variance of a function.
In many applications in that and other papers, the corresponding function is the indicator function of an event. In our theorem we restrict to such
functions.

\end{rem}

In Section \ref{sect-appl}
we illustrate Theorem \ref{Maint} by studying the
event that there is a horizontal crossing of an $R \times R$ box on the triangular lattice
in the site percolation model where exactly $k$ of the $n := R^2$
vertices are occupied. We show there
that the expected number of pivotal sites (and, consequently, the `discrete derivative' of the crossing probability
with respect to the fraction of occupied sites) at the value $k = R^2/2$ is larger than some positive
power of $R$, see Theorem \ref{thm-cross-piv}. The proof uses, besides Theorem \ref{Maint}, only a minimum of preliminaries from Bernoulli percolation (i.e. the
usual percolation model where the states of the vertices are independent of each other). It can be proved without using Theorem \ref{Maint}, but we don't
know a proof which neither uses Theorem \ref{Maint}, nor quite heavy results from Bernoulli percolation; see the comments and discussion in
Section \ref{sec:final}.

\end{subsection}
\begin{subsection}{Other related work} \label{rel-work}

The OSSS inequality, seen from right to left, gives a {\it lower} bound for the expectation of the number of pivotals,
and that is how it has been successfully used
in percolation theory and related fields, for instance in \cite{DRT19} and other papers mentioned above. Other well-known inequalities
for product measures,
which do not involve decision trees but also provide a lower bound for that expectation,
are the KKL inequality \cite{KKL}, and a related inequality
by Talagrand \cite{Tal94}. 

The KKL inequality (or, rather, a consequence of it) and Talagrand's inequality say,
roughly speaking, that, for product measures, the expected number of pivotals for the event $A$ is at
least some constant times the probability of $A$,
times ($1$ minus the probability of $A$) times $\log(1 / M)$, where $M$ is the maximum over all $e$ of the probability that $e$ is
pivotal. In situations where no suitable decision tree exists (or is known) but where some upper bound on the quantity $M$ is known, the
OSSS inequality is often useless while the KKL and Talagrand's inequalities still give a useful result. On the other hand, in specific
situations, where some suitable decision tree does exist, OSSS can be substantially stronger than KKL and Talagrand's inequality.

While the proof of KKL (and of Talagrand's inequality) has an analytic/algebraic flavour
(Fourier expansion, hypercontractivity), the OSSS inequality is, essentially, proved
from suitable coupling arguments which are more `probabilistic' in nature. 

The KKL inequality was generalised to $k$-out-of-$n$ measures (and similar measures on state spaces with a larger `alphabet', e.g. $\{0,1,2\}^n$) 
by O'Donnell and Wimmer \cite{OWi13} (see also \cite{FOWu22}).

Finally, we remark here that the paper \cite{BJ12} extends yet another inequality from product measures to $k$-out-of-$n$ measures. However, that inequality
(and its proof) are very different in nature from OSSS.

\end{subsection}
\begin{subsection}{Organization of the paper} \label{orga}
In Section \ref{Prelim} we introduce a `construction', a suitable sequence of strings ${\bf Z}^{(1)}, \cdots {\bf Z}^{(n)}$ of $0$'s and $1$'s. This
construction is later, in Section \ref{subsProofMaint}, used in the proof of our main result, Theorem \ref{Maint}.
It is specifically designed for the case of $k$-out-of-$n$ distributions and differs 
substantially from that in \cite{DRT19}, which was based on an encoding in terms of independent, uniformly on the interval $[0,1]$ distributed
random variables. Section \ref{sect-rem-proof} explains why we needed a different construction.

In Section \ref{sect-appl} we use Theorem \ref{Maint} in the study of  box-crossing probabilities for a 
percolation model where a fixed number of vertices is occupied: see Theorem \ref{thm-cross-piv} in that section.
This is meant as an illustration of how Theorem \ref{Maint} can be used, not as a major application. No specific percolation knowledge will be assumed in that section:
the proof of Theorem \ref{thm-cross-piv} is almost self-contained, apart from a few mild preliminaries in Section \ref{sect-ingr-thm-cross}. 
Section \ref{sec:final} gives several remarks concerning, among other things, another (potential) way to prove Theorem \ref{thm-cross-piv}.
\end{subsection}

\end{section}

\begin{section}{Proof of the main result}\label{sect-Proof-main}
\begin{subsection}{A construction and preliminary steps used in the proof}\label{Prelim}
We first introduce some more definitions and notation.
Let, as before, $E$ be a set with $n$ elements, and let $0 < k < n$. 
We use the notation $\Om_{k,n}$ for the set of all $x \in \{0,1\}^n$ with exactly $k$ $1's$.
Let $x$ and $y$ be two strings $\in \Om_{k,n}$. We call a point $e \in E$ a {\it disagreement point} (for the pair $(x,y)$)
if $x_e \neq y_e$. Further, we call it a disagreement point of {\it type $(1,0)$} if $x_e = 1$ and $y_e = 0$, and of {\it type $(0,1)$} in the reverse case.
Observe that the number of type 1 disagreement points for $(x,y)$ is equal to the number of type 2 disagreement points. 
We will denote the number of disagreement points for $(x,y)$ by $d(x,y)$.

A {\it matching of disagreement points} (for the pair $(x,y)$) is a map $\sigma$ from the set of disagreement points to itself with the following
two properties: If $\sigma(e) = \sigma(f)$, then $\sigma(f) = \sigma(e)$. And, if $e$ is a disagreement point of type $(1,0)$, then $\sigma(e)$ is of type $(0,1)$.

Often, instead of `matching of disagreement points' we simply say `matching'. Observe that if $\sigma$ is a matching for $(x,y)$, and $e$ is a disagreement point
for $(x,y)$, then the string $x^{(e,\sigma(e))}$ is in $\Om_{k,E}$. \\
For convenience, we extend the map $\sigma$ to the entire set $E$ by defining $\sigma(f) = f$ for all points $f \in E$ with $x_f = y_f$
(these points are called {\it singletons}). 

In the proof of our main result we often will consider, for a given pair $(x,y)$, with $x, y \in \Om_{k,n}$, a random map, typically denoted by $\bf \sigma$,
uniformly drawn from the set of all matchings for the pair $(x,y)$. We call it a {\it uniform matching} (for the pair $(x,y)$).

\medskip
Now let $A$ be an event (i.e. $A \subset \{0,1\}^E$) and $T = (e_1, \phi)$ a decision tree.
Let $\bf X$ and $\bf Y$ be independent random strings, each drawn from the distribution $P_{k,n}$.
Further, let $\bf \sigma$ be a uniform matching for $(\bf X, \bf Y)$. (This means that, given ${\bf X} = x$ and ${\bf Y} = y$, $\bf \sigma$
is drawn uniformly from the set of all matchings for the pair $(x,y)$). 

Define inductively, for $t \geq 1$, (with $\phi$ and $e_1$ as in the definition of a decision tree $T$ in the paragraphs preceding Theorem \ref{Maint}),

\begin{equation} \label{def-e-induct-n}
{\bf e}_t = 
\begin{cases}
e_1 & \text{ if } t=1, \\
\phi_t({\bf e}_{[t-1]}, {\bf X}_{{\bf e}_{[t-1]}}) & \text{ if } t > 1,
\end{cases}
\end{equation}

and (equivalently to the stopping time in \eqref{def-tau}), \\
\begin{equation}\label{def-tau-2-n}
\tau := \min\{t \geq 1 \, : \, \forall x \in \{0,1\}^E \text{ with } x_{{\bf e}_{[t]}} = {\bf X}_{{\bf e}_{[t]}}, \,
I_A(x) = I_A({\bf X})\}.
\end{equation}

Now we are ready to define a sequence of strings ${\bf Z}^{(0)}, {\bf Z}^{(1)},\cdots, {\bf Z}^{(n)}$.
We take ${\bf Z}^{(0)} = {\bf X}$. Then, step by step, we make ${\bf Z}^{(i)}$ closer to ${\bf Y}$ in the following (somewhat informally described)
way. At step 1 we replace the ${\bf Z}^{(0)}$-value (i.e. the ${\bf X}$-value) at the point ${\bf e}_1$ ($= e_1$) by the ${\bf Y}$-value at that point.
To keep the string in $\Om_{k,n}$ we of course have to make another change if ${\bf X}_{e_1} \neq {\bf Y}_{e_1}$, and do that at the point ${\bf \sigma}(e_1)$.
Similarly, at step 2 we replace the value at the point ${\bf e}_2$ in the string ${\bf Z}^{(1)}$ by the ${\bf Y}$-value at that point, and do the same at 
${\bf \sigma}({\bf e}_2)$. (Note that if ${\bf e}_2 = {\bf \sigma}(e_1)$ or if ${\bf \sigma}({\bf e}_2) = {\bf e}_2$, nothing changes at that step).
We continue doing this up to time ${\bf \tau}$ (after which no changes are made anymore).

Here is the formal definition:
\begin{eqnarray}\label{Z-def}
&\, & {\bf Z}^{(0)} := {\bf X}. \\
&\, & \text{For } 0 < j \leq \tau, \, {\bf Z}^{(j)}_e := 
\begin{cases}
{\bf Y}_e, & \text{ if } e = {\bf e}_j \text{ or } \sigma({\bf e}_j), \\ \nonumber
{\bf Z}^{(j-1)}_e, & \text{ otherwise. }
\end{cases} \\ \nonumber
&\, &  \text{For } j > \tau, \, {\bf Z}^{(j)} := {\bf Z}^{(j-1)}. \nonumber
\end{eqnarray}

From the above definition one easily gets the following:

\begin{obs}\label{Z-def-obs}
For $j \leq \tau$, we have 
\begin{equation}
{\bf Z}_e^{(j)} = 
\begin{cases}
{\bf Y}_e & \text{ for all } e \in \{{\bf e}_1, \cdots, {\bf e}_j\} \cup \{\sigma({\bf e}_1), \cdots, \sigma({\bf e}_j)\}, \\ \nonumber
{\bf X}_e & \text{ for all other } e \in E. \nonumber
\end{cases}
\end{equation}
Further, for $j > \tau$ we have ${\bf Z}^{(j)} = {\bf Z}^{(\tau)}$.
\end{obs}

From the above Observation we see that, given ${\bf X}_{{\bf e}_{[t]}}$ (i.e. given ${\bf e}_{[t]}$ and the ${\bf X}$-values on the points in ${\bf e}_{[t]}$) with
$t \leq \tau$, ${\bf Z}^{(t)} \equiv {\bf Y}$ on the points of ${\bf e}_{[t]}$.
Hence, noting that the event $\{t \leq \tau\}$ is measurable w.r.t. ${\bf X}_{{\bf e}_{[t]}}$ (and recalling that ${\bf X}$ and ${\bf Y}$ are independent),
the conditional distribution of the ${\bf Z}^{(t)}$-values on ${\bf e}_{[t]}$, given the above mentioned information, is $P_{k,E}$ (restricted to ${\bf e}_{[t]}$).
Further, given ${\bf X}_{{\bf e}_{[t]}}$ (still with $t \leq \tau$) and ${\bf Z}^{(t)}_{{\bf e}_{[t]}}$, the only information we have on the ${\bf Z}^{(t)}$-values
outside ${\bf e}_{[t]}$ is the {\it number} of $0$'s and $1$'s, not their locations. In other words, the conditional distribution of
${\bf Z}^{(t)}_{E \setminus {\bf e}_{[t]}}$ is $P_{k - |{\bf Z}^{(t)}_{{\bf e}_{[t]}}|, \, E \setminus {\bf e}_{[t]}}$. 
Combining these statements we get 

\begin{lem}\label{Z-lem-1} 
(a) The conditional distribution of ${\bf Z}^{(t)}$, given ${\bf X}_{{\bf e}_{[t]}}$ with $t \leq \tau$, is $P_{k,E}$. \\
(b) ${\bf Z}^{(n)}$ is independent of ${\bf X}_{{\bf e}_{[\tau]}}$ and has distribution $P_{k,E}$. In particular, ${\bf Z}^{(n)}$ and 
$I_A({\bf X})$ are independent.
\end{lem}

\begin{proof} Part (a) follows from the paragraph preceding the lemma. The first statement in part (b) follows from part (a)
and the property (in the definition of the ${\bf Z}$-strings)
that ${\bf Z}^{(t)} = {\bf Z}^{(\tau)}$ for all $t \geq \tau$.
The second statement follows from the first because $I_A({\bf X})$ is measurable with respect to ${\bf X}_{{\bf e}_{[\tau]}}$.
\end{proof}

\end{subsection}

\begin{subsection}{Proof of Theorem \ref{Maint}}\label{subsProofMaint} 
Let $c_1$ be a postive number. (A more precise choice of $c_1$ will be made later).
By Lemma \ref{Z-lem-1} we have
\begin{eqnarray}\label{eq-terms-bnd}
\nonumber
&\, & 2 P_{k,E}(A) (1 - P_{k,E}(A)) = \P(\text{ exactly one of } {\bf X}  \text{ and } {\bf Z}^{(n)} \text{ is in } A) \\ 
&\, & = \text{ TERM(1) } + \text{ TERM(2) },
\end{eqnarray}
where
\begin{equation}\label{q-T1-def}
\text{ TERM(1) } =
\P(\text{ exactly one of } {\bf X} \text { and } {\bf Z}^{(n)} \text{ is in } A, \text{ and } d({\bf X}, {\bf Y}) < c_1 n),
\end{equation}

and 
\begin{equation}\label{q-T2-def}
\text{TERM(2)} = \P(\text{ exactly one of } {\bf X} \text { and } {\bf Z}^{(n)} \text{ is in } A, \text{ and } d({\bf X}, {\bf Y}) \geq c_1 n). 
\end{equation}

From now on we will use (with $x, y \in \{0,1\}^E$ and $B \subset \{0,1\}^E$) the notation $!(x,y,B)$ for
``exactly one of $x$ and $y$ is in $B$".

Recall from the construction and definitions in Section \ref{Prelim}, that ${\bf Z}^{(t)} = {\bf Z}^{(\tau)}$ for all $t \geq \tau$.
Hence, we get

\begin{equation}\label{q-T1-r}
\text{ TERM(1) } =
\sum_{t=1}^n \P(\tau = t, \, !({\bf X }, {\bf Z}^{(t)}, A), \, d({\bf X}, {\bf Y}) < c_1 n),
\end{equation}

and

\begin{equation}\label{q-T2-r}
\text{TERM(2)} \leq \sum_{t=1}^n \P(t \leq \tau, \, !({\bf Z}^{(t-1)}, {\bf Z}^{(t)}, A), \, d({\bf X}, {\bf Y}) \geq c_1 n).
\end{equation}

Further, note that if ${\bf e}_t = e$ and exactly one of ${\bf Z}^{(t-1)}$ and ${\bf Z}^{(t)}$ is in $A$, then $(e, \sigma(e))$ is a 
pivotal pair for ${\bf Z}^{(t)}$. Hence,

\begin{equation}\label{q-T2-rr}
\text{TERM(2)} \leq \sum_{t=1}^n \sum_e \sum_f \P(t \leq \tau, \, {\bf e}_t = e, \,
(e,f) \text { is a piv. pair in } {\bf Z}^{(t)}, \, \sigma(e) = f, \,
d({\bf X}, {\bf Y}) \geq c_1 n).
\end{equation}

To handle \eqref{q-T1-r} and \eqref{q-T2-rr} further, we need the following key lemma.
Recall that the event $\{t \leq \tau\}$ is ${\bf X}_{\bf{e}_{[t]}}$-measurable (and even ${\bf X}_{\bf{e}_{[t-1]}}$-measurable).

\begin{lem}\label{lem-key}
Suppose ${\bf X}_{\bf{e}_{[t]}}$ is given and is such that $t \leq \tau$, {\it and} that the set of all matching pairs for $({\bf X}, {\bf Y})$ 
is given.
The conditional distribution of ${\bf Z}^{(t)}$ given the above information
is equal to the conditional distribution of ${\bf Y}$ given that same information.
\end{lem}

In fact, we prove something stronger, namely the following claim:

\begin{claim}\label{claim-key}
The lemma above still holds if we add, to the information conditioned on, the values ${\bf X}_s$ of all singletons $s \not\in {\bf e}_{[t]}$.
\end{claim}

\begin{proof} (of Claim \ref{claim-key})
For convenience we assume a predetermined (deterministic) ordering, simply denoted by `$<$', on the set $E$.

The information on which we condition provides the following knowledge about ${\bf Z}^{(t)}$ and $\bf Y$:
\begin{enumerate}
\item The exact ${\bf Y}$- and ${\bf Z}^{(t)}$-values at all points $g, h$ with $g \in {\bf e}_{[t]}$ and $h = \sigma(g) \neq g$. (Namely,
${\bf Z}^{(t)}_g = {\bf Y}_g = 1 - {\bf X}_g$ and ${\bf Z}^{(t)}_h = {\bf Y}_h = 1 - {\bf X}_h = {\bf X}_g).$

\item The exact ${\bf Y}$- and ${\bf Z}^{(t)}$-values at every singleton $s$. (Namely, ${\bf Z}^{(t)}_s = {\bf Y}_s = {\bf X}_s$).
\end{enumerate}

So we are left with the pairs of points $g ,h \not\in {\bf e}_{[t]}$ for which $\sigma(g) = h$ and $g < h$.
These pairs of points are given by the information conditioned on, but their precise ${\bf X}$-, ${\bf Y}$- and ${\bf Z}^{(t)}$- values are not.
However, it is not hard to see from the basic properties of the distribution $P_{k,E}$ that, conditioned on the information we 
have, for each such pair $(g,h)$ the following holds:
$({\bf X}_g, {\bf X}_h)$ is either $(0,1)$ or $(1,0)$, both with (conditional) probability $1/2$, 
independently of the other such pairs.
Finally, using that for each such pair $(g,h)$, 
$${\bf Z}^{(t)}_g = {\bf X}_g = 1 - {\bf Y}_g, \text{ and } {\bf Z}^{(t)}_h = {\bf X}_h = 1 - {\bf Y}_h,$$
it follows that, conditioned on the information mentioned in the Claim, also $({\bf Z}^{(t)}_g, {\bf Z}^{t)}_h)$ is $(0,1)$ or $(1,0)$, both with 
probability $1/2$, independent of the other pairs. And, similarly, this also holds for the ${\bf Y}$-values at these pairs.

Concluding, we get that ${\bf Y}$ and ${\bf Z}^{(t)}$ indeed have the same conditional distribution, given the information mentioned in the Claim.
This completes the proof of Claim \ref{claim-key}.
\end{proof}

\begin{proof} (of Lemma \ref{lem-key}).
Since the information conditioned on in the lemma is more coarse than that conditioned on in Claim \ref{claim-key}, the lemma follows
immediately.
\end{proof}

Now we apply Lemma \ref{lem-key} to further upper bound TERM(1) and TERM(2). As to the former, we first split the r.h.s.
of \eqref{q-T1-r} in two parts, giving:

\begin{eqnarray}\label{q-T1-rr}
\nonumber
\text{ TERM(1) } = \sum_{t=1}^n \P(\tau = t, \, {\bf X}\in A, \, {\bf Z}^{t} \not \in A, \, d({\bf X}, {\bf Y}) < c_1 n) \\ 
+ \sum_{t=1}^n \P(\tau =t, \, {\bf X} \not \in A, \, {\bf Z}^{(t)} \in A, \, d({\bf X}, {\bf Y}) < c_1 n).
\end{eqnarray}
Since, for each $t$, the event $\{\tau = t, \, {\bf X}\in A, \, d({\bf X}, {\bf Y}) < c_1 n\}$ is measurable w.r.t. the information conditioned on
in Lemma \ref{lem-key} (and is contained in the event $\{t \leq \tau\}$), that lemma allows us to replace ${\bf Z}^{t}$ by ${\bf Y}$ in
the first summation in the r.h.s. of \eqref{q-T1-rr}. And, analogously, this also holds for the second summation. So we get

\begin{eqnarray}\label{q-T1-rrr}
\nonumber
&\, & \text{TERM(1)} \\ \nonumber
&\, & = \sum_{t=1}^n \P(\tau = t, \, {\bf X} \in A, \, {\bf Y} \not\in A, \, d({\bf X}, {\bf Y}) < c_1 n) \\ \nonumber
&\, & + \sum_{t=1}^n \P(\tau =t, \, {\bf X} \not \in A, \, {\bf Y} \in A, \, d({\bf X}, {\bf Y}) < c_1 n) \\ 
&\, & = \P( !({\bf X}, {\bf Y}, A), \, d({\bf X}, {\bf Y}) < c_1 n).
\end{eqnarray}

We handle the last line of \eqref{q-T1-rrr} as follows, using symmetry (see also Remark \ref{rem-neg-c} below):

\begin{eqnarray}\label{q-T1-rrrr}
\nonumber
&\, &  \P( !({\bf X}, {\bf Y}, A), \, d({\bf X}, {\bf Y}) < c_1 n) \\ \nonumber
&\, & = 2 \P({\bf X} \in A, \, {\bf Y} \not \in A, \, d({\bf X}, {\bf Y}) < c_1 n) \\ \nonumber
&\, & \leq 2 \min\left(\P({\bf X} \in A, \, d({\bf X}, {\bf Y}) < c_1 n), \P({\bf X} \not\in A, \, d({\bf X}, {\bf Y}) < c_1 n)\right) \\
\nonumber
&\, &  = 2 \P( d({\bf X}, {\bf Y}) < c_1 n) \, \min(P_{k,E}(A), 1 - P_{k,E}(A)) \\ 
&\, &  \leq 4 P_{k,E}(A) \, (1 - P_{k,E}(A)) \, \P( d({\bf X}, {\bf Y}) < c_1 n),
\end{eqnarray}
where the third step uses that the random variables ${\bf X}$ and $d({\bf X}, {\bf Y})$ are independent (which follows
easily from the permutation invariance of $P_{k,E}$ and the independence of ${\bf X}$ and ${\bf Y}$).

So, combining \eqref{q-T1-rrr} and \eqref{q-T1-rrrr}, we finally get the following upper bound for TERM(1).

\begin{equation}\label{q-T1-rf}
\text{ TERM(1) } \leq 4 P_{k,E}(A) \, (1 - P_{k,E}(A)) \, \P(d({\bf X}, {\bf Y}) < c_1 n).
\end{equation}

\begin{rem}\label{rem-neg-c}
One might at first sight intuitively think that the two events in the last line of \eqref{q-T1-rrr} are
negatively correlated, which would allow us to make the computations above a bit more elegant and stronger: it would lead
to an upper bound for TERM(1) half that in \eqref{q-T1-rf} (and hence a factor $\P(d({\bf X}, {\bf Y}) < c_1 n)$ smaller
than the l.h.s. of \eqref{eq-terms-bnd}). However, there is a simple counterexample, showing that that negative correlation
is not true in general.
\end{rem}

Now we return to TERM(2) in \eqref{eq-terms-bnd}, and start with the upper bound \eqref{q-T2-rr}.
Note that, apart from the part ``$(e,f)$ is pivotal for ${\bf Z}^{(t)}$", the event in the r.h.s. of \eqref{q-T2-rr} is
measurable with respect to the information conditioned on in Lemma \ref{lem-key}.
Hence, by that lemma, we can replace ${\bf Z}^{(t)}$ by ${\bf Y}$ in \eqref{q-T2-rr}, and thus get

\begin{equation}\label{q-T2-rrr}
\text{TERM(2)} \leq \sum_{t=1}^n \sum_e \sum_f
\P(t \leq \tau, \, {\bf e}_t = e, \, (e,f) \text{ is piv. in } {\bf Y}, \, \sigma(e) = f, \, d({\bf X}, {\bf Y}) \geq c_1 n).
\end{equation}

Now, trivially,
\begin{eqnarray} \label{q-T2-rrrt}
\nonumber
&\, & \text{The probability in the r.h.s. of \eqref{q-T2-rrr}} \\ \nonumber
&\, & = \P(t \leq \tau, \, {\bf e}_t = e, \, (e,f) \text{ is piv. in } {\bf Y}) \\ 
&\, & \times \P(\sigma(e) = f, \, d({\bf X}, {\bf Y}) \geq c_1 n \, | \,
t \leq \tau, \, {\bf e}_t = e, \, (e,f) \text{ is piv. in } {\bf Y})
\end{eqnarray}

The first factor in the r.h.s. of \eqref{q-T2-rrrt} is equal to 
$$\P(t \leq \tau, \, {\bf e}_t = e) \, \P((e,f) \text{ is piv. in } {\bf Y}),$$
because $\{t \leq \tau, \, {\bf e}_t = e\}$ is ${\bf X}$-measurable and ${\bf X}$ and ${\bf Y}$ are independent.

The second factor in the r.h.s. of \eqref{q-T2-rrrt} is, obviously, 
$$ \leq \P(\sigma(e) = f \, | \, d({\bf X}, {\bf Y}) \geq c_1 n, \, t \leq \tau, \, {\bf e}_t = e, \, (e,f) \text{ is piv. in } {\bf Y}), $$
which in turn is at most $\frac{2}{c_1 n}$ because of the following. Let $0 \leq m \leq n$ be some integer.
If the pair $({\bf X}, {\bf Y})$ would be exactly given, and would have exactly $m$ disagreement points, the probability that a given
pair of points $g$, $h$, with $g \neq h$, is a matching pair would be at most $2/m$ (in fact it would be equal to $2/m$ if 
${\bf X}_g = {\bf Y}_h = 1 - {\bf Y}_g = 1 - {\bf X}_h$), and equal to $0$ otherwise).

So, putting the above things together, we finally get the following upper bound for TERM(2).

\begin{eqnarray}\label{q-T2-rrrf}
\nonumber
&\, & \text{TERM(2) } \\ \nonumber
&\, &  \leq \sum_{t=1}^{n} \, \sum_{e \in E} \, \sum_{f \neq e} \frac{2}{c_1 n} \, \P(t \leq \tau, \, {\bf e}_t = e) \, 
\P((e,f) \text{ is piv. in } {\bf Y}) \\ \nonumber 
&\, & \leq \sum_{t=1}^{n} \, \sum_{e \in E} \, \sum_{f \in E} \frac{2}{c_1 n} \,  \P(t \leq \tau, \, {\bf e}_t = e) \,
(I(e) + I(f)) \\ \nonumber
&\, &  = \frac{2}{c_1} \, \sum_{e \in E} I(e) \, \sum_{t=1}^n \P(t \leq \tau, \, {\bf e}_t = e) + 
\frac{2}{c_1} \, \frac{1}{n} \sum_{e \in E} \sum_{t=1}^{n} \P(t \leq \tau, \, {\bf e}_t = e) \, \sum_{f \in E} I(f) \\ 
&\, &  = \frac{2}{c_1} \left( \sum_{e \in E} I(e) \delta_e + \sum_{e \in E} I(e) \bar\delta \right),
\end{eqnarray}
where the second inequality holds because the event $A$ is increasing.

Combining \eqref{eq-terms-bnd}, \eqref{q-T1-rf} and \eqref{q-T2-rrrf}, we get

\begin{eqnarray}\label{eq-terms-bnd2}
\nonumber
&\, & 2 P_{k,E}(A) (1 - P_{k,E}(A)) \\ \nonumber
&\, & \leq 4 P_{k,E}(A) (1 - P_{k,E}(A))  \, \P(d({\bf X}, {\bf Y}) < c_1 n) \\
&\, & + \frac{2}{c_1} \left( \sum_{e \in E} I(e) \delta_e + \sum_{e \in E} I(e) \bar\delta \right).
\end{eqnarray}
We are now close to the completion of the proof of Theorem \ref{Maint}.

\begin{subsubsection}{Completion of the proof of Theorem \ref{Maint}}\label{cpl-proof-maint}

We first restrict to the special case in part (a) of the theorem.
So we take $n := |E|$ even, and $k = n/2$. It is well-known and easy to see that for this case 
$\E(d({\bf X}, {\bf Y})) = n/2$, and that, if we take $c_1$ smaller than $1/2$, then 
$\P(d({\bf X}, {\bf Y}) < c_1 n)$ converges to $0$ as $n \rightarrow \infty$. (By standard concentration-inequalities techniques
it even follows that this convergence is exponentially fast in $n$).
So, choosing $c_1$ equal to (say) $1/4$, we can take $N$ so that $\P(d({\bf X}, {\bf Y}) < c_1 n) \leq 1/4$ for all $n \geq N$.
Hence, by \eqref{eq-terms-bnd2} we get that, for all $n \geq N$,
\begin{eqnarray}\label{eq-prf(a)-f}
\nonumber
&\, & P_{\frac{n}{2}}(A) (1 - P_{\frac{n}{2}}(A)) \leq  \\ \nonumber
&\, & \frac{2}{c_1} \, \left( \sum_{e \in E} I(e) \delta_e + \sum_{e \in E} I(e) \bar\delta \right) \\ 
&\, & = 8 \, \left( \sum_{e \in E} I(e) \delta_e + \sum_{e \in E} I(e) \bar\delta \right).
\end{eqnarray}
Since there are only finitely many $n < N$, and for each $n$ only finitely many possibilities for the event $A$ and the
decision tree $T$, we  obtain part (a) of Theorem \ref{Maint} by replacing the factor $8$ in the upper bound above by a sufficiently large
constant $C$. This completes the proof of part (a) of the theorem. 

As to the proof of part (b) of the Theorem, we can modify (generalise) the argument for part (a) above in a straightforward way by 
choosing $c_1$ conveniently, depending on $\varepsilon$.

This completes the proof of Theorem \ref{Maint}.
%
 
\end{subsubsection}

\end{subsection}

\begin{subsection}{Remark on the construction in the proof}\label{sect-rem-proof}

The sequence of strings ${\bf Z}^{(t)}$, $0 \leq t \leq n$, constructed in Section \ref{sect-Proof-main} for the proof of Theorem \ref{Maint},
was tailor-made for $k$-out-of-$n$ distributions: it doesn't make sense for other distributions. 
In an earlier version of this paper we used the construction of \cite{DRT19}. That construction, fruitfully used in \cite{DRT19} for monotonic measures,
makes {\it in principle} sense for general distributions, which doesn't mean that it always leads to a useful inequality.

Our computations in the mentioned earlier version of the paper, using that construction for $k$-out-of-$n$ distributions,
gave an extra factor of order $\log n$ compared with the upper bound \eqref{eq-maint-a} (and \eqref{eq-maint-b}) in the current version.
We initially thought that
that extra factor was caused by a crude estimate at a later stage in the computations. However, as we briefly point out below, there is a simple example
of an event $A$ where the $\log n$ factor comes up inevitably by using that construction.

Central in the construction in \cite{DRT19} is the idea of encoding a string ${\bf X} = ({\bf X}_1, \cdots, {\bf X}_n)$ of (possibly dependent) 
$0-1$ valued random variables in terms of a sequence ${\bf U} = ({\bf U}_1, \cdots {\bf U}_n)$ of independent,
uniformly distributed on the interval $[0,1]$, random
variables. This encoding is informally as follows:
Let $\mu$ denote the distribution of ${\bf X}$. Let $u_1, \cdots, u_n$ be numbers between $0$ and $1$. These numbers can be seen as an input string in
the coding procedure. The output string $(x_1, \cdots, x_n)$,
denoted by $F^{\mu}(u_1, \cdots, u_n)$ is the following: If $u_1 < \mu(\omega_1 = 0)$, take $x_1 = 0$, otherwise
take $x_1 = 1$. Next, if $u_2 < \mu(\omega_2 = 0 \, | \, \omega_1 = x_1)$, take $x_2 = 0$, otherwise take $x_2 = 1$. Etcetera. 
It is well-known and easy to check that $F^{\mu}({\bf U})$ has distribution $\mu$. 
As shown (and used) in \cite{DRT19} this is still true (under a natural condition) if the order in which output values are assigned
to the points $1, \cdots, n$ is
not deterministic: the identity of the $j$th point to which an output value is assigned may depend on the identity of the 
first $j - 1$ points and the output values assigned to them. However, the example below involves only the case with deterministic order:

Consider our main result, Theorem \ref{Maint}, with  $n$ even, $E = \{1, \cdots, n\}$, $k = n/2$,
$A = \{\omega_n = 1\}$, and $T$ the (deterministic) decision tree where the points of $E$ are examined in the order $1, 2, 3, \cdots$. 
Note that then $\tau = n$.
The construction in \cite{DRT19} (see Section 2 in that paper, in particular the first lines of page 85,
with $f$ the indicator function of the event $A$ in our example) gives, for this simple case:

\begin{eqnarray}\label{eq-log-ex}
&\, &  2 P_{\frac{n}{2},n}(A) (1 - P_{\frac{n}{2},n}(A)) \\ \nonumber
&\, &  \leq \sum_{t=1}^n \P \left( !(F^{\mu}({\bf V}_1, \cdots, {\bf V}_{t-1}, {\bf U}_t, \cdots {\bf U}_n),
F^{\mu}({\bf V}_1, \cdots, {\bf V}_t, {\bf U}_{t+1}, \cdots {\bf U}_n), A \right),
\end{eqnarray}
where $\mu = P_{\frac{n}{2},n}$, and ${\bf U}_1, \cdots, {\bf U}_n, {\bf V}_1, \cdots, {\bf V}_n$ are independent random variables, 
uniformly distributed on the interval $[0,1]$.
To show that the r.h.s. of \eqref{eq-log-ex} is of order $\log n$, we will use the following lemma, which can be proved by a quite straightforward 
induction argument.
\begin{lem} \label{exch-rule1}
Let $m$ be a positive integer and let $1 \leq k \leq m-1$. Let ${\bf U}_1, \cdots, {\bf U}_m$ be independent random variables,
uniformly distributed on $[0,1]$. Let ${\bf Z} = F^{P_{k,m}}(\bf U)$ and ${\bf Z'} = F^{P_{k+1,m}}({\bf U})$, 
with ${\bf U} = {\bf U}_1, \cdots, {\bf U}_m$. \\
The pair $({\bf Z}, {\bf Z'})$ has the following distribution.
For all $\alpha, \beta \in \{0,1\}^m$,

\begin{equation} \label{cases-ex-r1}
\P({\bf Z} = \alpha, {\bf Z'} = \beta) =
\begin{cases}
\frac{1}{m-k} P_{k,m}(\alpha), & \text{ if } |\alpha| =k, |\beta| = k+1 \text{ and } \beta \geq \alpha \\
0 , & \text{otherwise}
\end{cases}
\end{equation}

\end{lem}

\smallskip
{\bf Remark:} {\em Informally, the lemma says that, given that ${\bf Z} = \alpha$, ${\bf Z'}$ is `obtained' by randomly and uniformly
choosing one of the $0$'s in the
string $\alpha$ and
replacing it by $1$. So the pair $({\bf Z},{\bf Z'})$ represents the intuitively most natural coupling of $P_{k,m}$ and $P_{k+1,m}$.
}

Now we go back to the sum in the r.h.s. of \eqref{eq-log-ex}. Consider a term with $4 < t < n-4$. Note that the probability that 
${\bf V}_1, \cdots, {\bf V}_{t-1}$ are such that the number of
$1$'s at the positions $t, \cdots, n$ in the output string is $\in (\frac{n-t+1}{4}, \frac{3}{4} (n-t+1))$, 
is simply $P_{n/2,n}(|\omega_{[t,n]}| \in (\frac{n-t+1}{4}, \frac{3}{4} (n-t+1))$, which is larger than some constant $c_1 > 0$. 
Also note that, given that ${\bf V}_1, \cdots, {\bf V}_{t-1}$ satisfy the property in the previous sentence, the conditional
probability that $F^{\mu}({\bf V}_1, \cdots, {\bf V}_{t-1}, {\bf U}_t, \cdots {\bf U}_n)$ 
and $F^{\mu}({\bf V}_1, \cdots, {\bf V}_t, {\bf U}_{t+1}, \cdots {\bf U}_n)$ have value $1$ and $0$ at position $t$ 
respectively and $F^{\mu}({\bf V}_1, \cdots, {\bf V}_{t-1}, {\bf U}_t, \cdots {\bf U}_n)$ has value $0$ at position $n$,
is at least some constant $c_2 > 0$.
Because of this, and by using Lemma \ref{exch-rule1} with $m = n-t$, (and noting that for any $k$ the factor $1/(m-k)$
in \eqref{cases-ex-r1} is thus, in the current context,
at least $1/(n-t)$) it follows that the mentioned term in the r.h.s. of \eqref{eq-log-ex} is at least
$c_1 \times  c_2 \times \frac{1}{n-t}$. Summing this over all $t \in (4, n - 4)$, shows that the
r.h.s. of \eqref{eq-log-ex} is larger than some constant times $\log n$.
On the other hand, for this special example the r.h.s of \eqref{eq-maint-a} in our main result is (using that $I(e) = 0$ for
all $e \neq n$ in this example) bounded by a constant. This example explains why we needed a different construction from that in \cite{DRT19} 
to get rid of the $\log n$ factor.

\end{subsection}
\end{section}

\begin{section}{Box crossing probabilities in a percolation model with a fixed number of occupied vertices} \label{sect-appl}
Now we illustrate our main result, Theorem \ref{Maint}, by using it in the study of a percolation model on a
box where the number of occupied vertices is 
fixed. We will also use a few basic results from the literature on the `standard' percolation model (Bernoulli percolation),
where the states of the vertices (or edges) are {\it independent} of each other.
See e.g. \cite{Gr99} and \cite{Gr18} for a general introduction to Bernoulli percolation, and many results and references.

Consider an $R \times R$ box in the triangular lattice.  
More precisely, in terms of the standard embedding of this lattice in the plane (which we identify with the set of complex numbers $\C$),
the box we consider is the graph
with vertex set $V_R :=  \{x + y \exp(i \pi / 3)\, : \, x,y \in \Z, \, 0 \leq x, y \leq R-1\}$,
and where two vertices $v$ and $w$ share an edge iff $|v-w| = 1$.
Each vertex can be {\it vacant} (which corresponds with value $0$) or {\it occupied} (value $1$).

Let $A_R$ denote the event that there is an occupied path which crosses the box horizontally.
In ordinary (Bernoulli) percolation, where the vertices are independently occupied with probability $p$ and vacant with probability $1-p$ (the
corresponding distribution will be denoted by $P_p$), 
it is well-known that $P_{\frac{1}{2}}(A_R) = 1/2$ (which follows from a simple symmetry
argument), and that the expected number of pivotal vertices (and hence, by the well-known Margulis-Russo formula, 
also $\frac{d}{d p} P_p(A_R) \mid_{p = 1/2}$)
grows at least as a power of $R$.
In fact, very sharp versions of this result are known, see \cite{SW01}. 

We will study the same event $A_R$ but now for the model where a fixed number, denoted by $k$, of the vertices is occupied. More precisely,
the probability measure is now $P_{k, R^2} = P_{k,V_R}$. In particular we will study the case where $R$ is 
even and $k=R^2/2$.

Again, from symmetry, $P_{\frac{R^2}{2}, R^2}(A_R) = 1/2$.
We will show an analog of the Bernoulli percolation result mentioned above, namely that
the `discrete derivative' (with respect to the fraction of $1's$) of the above probability is
again larger than a constant times a power of $R$.
In fact we will, using Theorem \ref{Maint}, prove Theorem \ref{thm-cross-piv} below, from which Corollary \ref{cor-cross-deriv} below follows
easily.
Recall the definition of $0$-pivotal from Section \ref{Intr} (two paragraphs below \eqref{PkE-def}).
\begin{rem}\label{rem-pivot-cr}
Although we don't use this in the proof of Theorem \ref{thm-cross-piv}, we note that (as is well-known), for this specific event $A_R$, a vertex
$v \in V_R$ is $0$-pivotal iff $v$ is vacant and there are four disjoint paths, each starting from a neighbour of $v$ to the boundary of the box: two occupied paths to 
the left and the right side, respectively, and two vacant paths to the top and the bottom side, respectively. See also Section \ref{sec:alt-proof}.
\end{rem}

\begin{thm} \label{thm-cross-piv}
Let $N_R^0$ denote the number of vertices that are $0$-pivotal for the event $A_R$, and let
$E_{\frac{R^2}{2}, R^2}(N_R^0)$ denote its expectation w.r.t. the distribution $P_{\frac{R^2}{2}, R^2}$.
There are $\alpha > 0$ and $C>0$ such that, for all even $R \geq 2$,
\begin{equation}\label{eq-thm4.1}
E_{\frac{R^2}{2}, R^2}(N_R^0) \geq C R^{\alpha}.
\end{equation}
\end{thm}

\begin{rem}\label{rem-th4.1}
{\em (a)} The proof of Theorem \ref{thm-cross-piv} in Section \ref{sec-proof-thm-5.1} uses our OSSS-type result for $k$-out-of-$n$ measures (Theorem \ref{Maint})
and a minimal amount of
knowledge of Bernoulli percolation:
essentially only RSW (and FKG), which are pre-1979 results. \\
{\em (b)} Theorem \ref{thm-cross-piv} (and a stronger version) can also be proved without using Theorem \ref{Maint}, but such a proof would (as far
as we know) require much heavier results from Bernoulli percolation, see Section \ref{sec:alt-proof}. \\
{\em (c)} We are not aware of earlier mathematically rigorous work on percolation with a fixed number of occupied vertices or edges.
In the physics literature, such models have been studied and compared with Bernoulli percolation: see the paper \cite{HBD} where heuristic
predictions are given for finite-size corrections of various quantities. \\
{\em (d)}  Analogs of Theorem \ref{thm-cross-piv} and Corollary \ref{cor-cross-deriv}
for {\it bond} percolation on the {\it square} lattice can be proved in
a very similar way.
\end{rem}

\begin{cor} \label{cor-cross-deriv}
For all even $R$, 

\begin{equation}\label{eq-cross-piv}
\frac{P_{\frac{R^2}{2} + 1, R^2}(A_R) - P_{\frac{R^2}{2}, R^2}(A_R)}{1/R^2} \geq 2 C R^{\alpha},
\end{equation}
with $C$ and $\alpha$ as in Theorem \ref{thm-cross-piv}.\\
\end{cor}

Note that the denominator in the l.h.s. of \eqref{eq-cross-piv} is the increase of the fraction of occupied vertices when the parameter 
changes from $R^2/2$ to $R^2/2 + 1$; so the l.h.s. of \eqref{eq-cross-piv} can indeed be interpreted as a `discrete derivative'.

The proof of Theorem \ref{thm-cross-piv} will be given in Section \ref{sec-proof-thm-5.1} (after stating a simple inequality and a standard ingredient
from Bernoulli percolation in Section \ref{sect-ingr-thm-cross}). We will now first show how
Corollary \ref{cor-cross-deriv} follows from that theorem. For that we use Observation \ref{lem-russo} below,
which holds for all increasing events and is an analog of the well-known Margulis-Russo formula for product measures.
Its proof (which we omit) uses a straightforward coupling of $P_{k+1,n}$ and $P_{k,n}$, is simpler than that of the Margulis-Russo formula, and is probably,
implicitly or explicitly, already in the literature.

\begin{obs} \label{lem-russo}
Let $n \geq 1$ and let $A \subset \{0,1\}^n$ be an increasing event. Let $N_A^0$ denote the number of vertices that are $0$-pivotal for $A$,
and $E_{k,n}(N_A^0)$ its expectation w.r.t. the distribution $P_{k,n}$. 
\smallskip
For all $k \leq n-1$,
\begin{equation} \label{russo-ineq}
P_{k+1,n}(A) - P_{k,n}(A) = \frac{1}{n-k} E_{k,n}(N_A^0).
\end{equation}

\end{obs}
%
%
\begin{proof} (of Corollary \ref{cor-cross-deriv} from Theorem \ref{thm-cross-piv}).
The corollary follows immediately from Theorem \ref{thm-cross-piv} by applying Observation \ref{lem-russo}, with
$n = R^2$ and $k = R^2/2$, to the event $A_R$.  
\end{proof}

\begin{subsection}{A percolation ingredient for the proof of Theorem \ref{thm-cross-piv}}\label{sect-ingr-thm-cross}
First we mention the following
quite obvious inequality comparing $P_{\frac{n}{2},n}$ and $P_{1/2}$:

\begin{obs}\label{triv-obs}
For all even $n \geq 2$ and all increasing events $A \subset \{0,1\}^n$, 
$$ P_{\frac{n}{2}, n}(A) \leq 2 P_{\frac{1}{2}}(A).$$
\end{obs}
\noindent (This Observation follows immediately from the simple facts that $P_{1/2}(| \omega | \geq n/2)$ and 
$P_{1/2} (A \,\, \vert \,\, | \omega | \geq n/2)) \geq P_{\frac{n}{2}, n}(A)$). \\

As said in Remark \ref{rem-th4.1}, the proof of Theorem \ref{thm-cross-piv} uses only a minimal amount of knowledge from Bernoulli percolation, namely,
FKG-Harris (the result that, for product measures,
increasing events are positively correlated) and RSW.
The form of RSW we use is that there is a $c_1 > 0$ such that for all $m \geq 1$,
\begin{equation}\label{rsw}
P_{\frac{1}{2}} (\exists \text{ an occupied vertical crossing of a given } 3 m \times m \text{ box }) < c_1.
\end{equation}
As is well-known, this (combined with FKG) easily implies that there is a $c_2 < 1$ such that,
for all $m\geq 1$, the probability (under $P_{\frac{1}{2}})$ that there is an
occupied path crossing the annulus between two concentric boxes, one of size $m \times m$ and the
other of size $3 m \times 3 m$, is smaller than $c_2$.
Since a path from $0$ to a point at distance $M$ from $0$ has to cross of order $\log M$ specific annuli of
the above shape, it follows immediately (and is
a classical result) that there are $c_3$, $c_4 > 0$ such that, for all $M \geq 1$,
\begin{equation}\label{1arm-bnd}
\pi(M) \leq c_3 \, M^{- c_4}, 
\end{equation}
$$\text{where } \pi(M) = P_{\frac{1}{2}}(\exists \text{ a path from } 0 \text{ to a point at distance } M \text{ from } 0).$$

\begin{rem} Around 2000 the result \eqref{1arm-bnd} has been dramatically improved (\cite{LSW02}), but \eqref{1arm-bnd}, together with
Theorem \ref{Maint}, is sufficient to prove Theorem \ref{thm-cross-piv}. 
\end{rem}
\end{subsection}

\begin{subsection}{Proof of Theorem \ref{thm-cross-piv} }\label{sec-proof-thm-5.1} 
We will use our main result (Theorem \ref{Maint})
and the existence of a suitable algorithm (decision tree) for checking the existence of an occupied horizontal crossing of the $R \times R$
box.
As in many applications of OSSS in the literature, we will average over a number of decision trees 
to obtain a desired `average' smallness of revealments.

For discovering a crossing from the left to the right side of a box, there is a well-known algorithm
involving a so-called {\it exploration path}. The construction of such a path
is best explained in terms of the graph obtained by representing each vertex of the triangular lattice as (the midpoint of)
a small hexagon in the hexagonal lattice (the dual of the triangular lattice). We colour a hexagon white
if its corresponding vertex in the triangular lattice is occupied, and black otherwise. The most common exploration path is a 
path (in the box) which starts
in a corner of the box, say the lower-right corner, and continues step-by-step in such a way that at each step
there is white on the right and black on the left. 
There is an occupied horizontal
crossing of the box if and only if the exploration path reaches the left side of the box before the top side.
So, to discover whether the event $A_R$ holds, only the colours of those vertices that are reached by the exploration path before it hits
the left or top side of the box, have to be revealed.

This algorithm in itself is not `suitable' yet, beacuse vertices close to the mentioned corner have a `high' probability to
be examined.
Schramm and Steif \cite{SS10} (see Section 4 of their paper) presented, in the context of Bernoulli percolation,
an adaptation of the algorithm, which is informally as follows. For each point $v_0$ (denoted by $p_0$ in their paper)
on the right side of the box,
consider a straightforward modification $\beta = \beta_{v_0}$ of the above exploration path to discover if there
is a horizontal occupied path from the left side of the box to the part of the right side above $v_0$.
And, similarly, a modification $\beta'_{v_0}$ to determine the existence of a horizontal
occupied path from the left side of the box to the part of the right side below $v_0$.
Together, these two exploration paths determine whether the event $A_R$ occurs. So, for each $v_0$ one now has a decision tree
for the event $A_R$.
It was shown in \cite{SS10} that,
for each vertex $w$ in the box, the average over all the above mentioned $v_0$'s of the corresponding revealment probabilities for $w$, is `small'.

Recall that the above mentioned work by Schramm and Steif concerns Bernoulli percolation. For the percolation model with fixed number of occupied vertices
we will consider the same decision trees (with exactly the same definition of $\beta_{v_0}$ and $\beta'_{v_0}$).
The revealment probabilities may of course differ from those in the Bernoulli case.
Below is an outline of the computation leading to an upper bound on these revealment probabilities,
where we focus on the differences and adaptations compared to that for Bernoulli percolation in \cite{SS10}.
See Section 4 of \cite{SS10} for more precise (deterministic) properties (and pictures) of the exploration paths.

Let $T_{v_0}$ denote the decision tree corresponding to the exploration paths $\beta_{v_0}$ and
$\beta'_{v_0}$. For brevity we will write $\P$ for $P_{\frac{R^2}{2}, R^2}$. 
We have, for each hexagon $H$,
\begin{equation}\label{eq:H-exam-ineq}
\P(H \text{ examined under } T_{v_0})  \leq \P(H \text{ on } \beta_{v_0}) + \P(H \text{ on } \beta'_{v_0}),
\end{equation}
where ``$H$ on $\beta$" means that at least one of the sides of $H$ is on $\beta$.

It is easy to see that if $H$ touches $\beta_{v_0}$ then there is a black or a white path from (a neighbour of) $H$ 
to distance at least $|z_0(H) - v_0| - K$ from $H$, where $z_0(H)$ is the closest point to $H$ on the right side
of the box, and $K$ is a (universal) constant. 
(In \cite{SS10} a considerably stronger statement was obtained, but for our purpose the one above is sufficient).
Hence,
\begin{eqnarray}\label{eq-conseq}
\nonumber
& \,& \P(H \text{ on } \beta_{v_0}) \\ \nonumber 
& \,&\leq \P(\exists \text{ a white or a black path from } H \text{ to distance }
|z_0(H) - v_0| - K \text{ from } H) \\ 
& \,& = 2 \P(\exists \text{ a white path from } H \text{ to distance }
|z_0(H) - v_0| - K \text{ from } H).
\end{eqnarray}
The same inequality also holds for $\beta'_{v_0}$, so together with \eqref{eq:H-exam-ineq} this gives
\begin{eqnarray} \nonumber
&\, & \P(H \text{ examined under } T_{v_0}) \\ \nonumber
&\, & \leq 4 \P(\exists \text{ white path from }
H \text{ to distance } |z_0(H) - v_0| - K \text{ from } H)).
\end{eqnarray}
Summing this over all $v_0$ on the right side of the box (there are roughly $R$ of these) and using
\eqref{1arm-bnd} and Observation \ref{triv-obs} gives, for every $H$,
\begin{eqnarray} \label{eq-box-reveal}
&\, & \sum_{v_0} \P(H \text{ examined under } T_{v_0}) \\ \nonumber 
&\, & \leq c_5 \sum_{m = 1}^R \P(\exists \text{ white path from }
H \text{ to distance } m \text{ from } H) \\ \nonumber
&\, & \leq 2 c_5 \sum_{m=1}^R c_3 m^{- c_4} \leq c_6 R^{1-c_4}, 
\end{eqnarray}
where $c_5$ and $c_6$ are constants $> 0$.
Hence, for some constant $c_7 > 0$,
\begin{equation}\label{eq-box-rev3}
\left[\text{The average over all } v_0 \text{ of } \P(H \text{ examined under } T_{v_0})\right] \leq c_7  R^{-c_4 }.
\end{equation}

Now for each $v_0$ we apply Theorem \ref{Maint} (with $E$ the set of vertices in the $R \times R$ box,
$A$ the above mentioned crossing event $A_R$, and $T=T_{v_0}$ ). So for each $v_0$ this gives an inequality of the
form \eqref{eq-maint-a}. Note that the l.h.s. of each of these inequalities is $1/4$. `Averaging' these inequalities and using \eqref{eq-box-rev3} yields
(with $c_8$ a constant $>0$)
$$1/4 \leq c_8 R^{- c_4} \, \sum_w \P(w \text{ is } 0-\text{pivotal for } A_R),$$
(where the sum is over all vertices $w$ in the $R \times R$ box)
and hence that the l.h.s. of \eqref{eq-thm4.1} is $\geq \frac{1}{4 c_8} R^{c_4}$.
This completes the proof of Theorem \ref{thm-cross-piv}. \qed \\

\end{subsection}

\begin{subsection}{Further comments on Theorem \ref{thm-cross-piv} and its proof}\label{sec:final}
\begin{subsubsection}{Sketch of an alternative proof of Theorem \ref{thm-cross-piv}}\label{sec:alt-proof}
As said, the proof of Theorem \ref{thm-cross-piv} in the previous subsection uses, besides our
OSSS-like result (Theorem \ref{Maint}), only a minimal amount
of (Bernoulli) percolation theory. We don't know a proof of Theorem \ref{thm-cross-piv} which neither uses the OSSS-like
result, nor more than only mild results from Bernoulli percolation.

There {\it is} a way to prove Theorem \ref{thm-cross-piv} (and even
a stronger version) without using OSSS,
but instead using relatively heavy results from Bernoulli percolation. We are not aware of such proof in the literature, and will give
an informal sketch of one (there are probably more ways): 

First, recall that one step (namely, the second to last inequality in \eqref{eq-box-reveal}) in the proof of Theorem \ref{thm-cross-piv}
uses an obvious comparison (Observation \ref{triv-obs}) between
$P_{\frac{1}{2}}$ and $\P$ ($=P_{\frac{R^2}{2}, R^2}$).
That comparison was used to bound the probability (under $\P$) 
that there is an occupied path from a given vertex to a region at a certain
distance from that vertex.
Also recall (see Remark \ref{rem-pivot-cr}) that the event that a given vertex $v$ is pivotal for $A_R$ is the 
event that there are four distinct `arms' with certain properties from 
neighbours of $v$ to the boundary of the $R \times R$ box.
It is natural to ask if the probability of that event
under the probability measure $\P$ can also be suitably compared with that under the measure $P_{\frac{1}{2}}$
(which would then yield another way to obtain Theorem \ref{thm-cross-piv}).

Such comparison can indeed be made, but it involves, roughly speaking, a change of the parameter (which was $1/2$) of the
Bernoulli process. Such change is needed to `control' the effect of the typical fluctuation of the number of occupied vertices
in the $R \times R$ box in the Bernoulli model. This fluctuation is of course of order of the square root of the number of vertices
in the box, i.e. of order $R$. This can be controlled (or `compensated') by a change of the Bernoulli parameter by roughly
$\frac{1}{R^2}$ $\times$ the mentioned typical fluctuation,
i.e.,  roughly $R/R^2 = 1/R$. The characteristic length scale of Bernoulli percolation at parameter $ 1/2 + \delta$, denoted by
$L(1/2 + \delta)$, is known to be of order at least $1/\delta$. (Much more is known but not needed here).

A central, technically complicated result in the celebrated paper \cite{K87} by Kesten (see also \cite{N08}) says that
at length scales below that characteristic length, the Bernoulli
model with the new parameter ($ 1/2 + \delta$) behaves typically like the critical model (i.e. the Bernoulli model with parameter $1/2$). 
A slightly more subtle version of Kesten's result (which is actually needed in the argument above, with an extra state of the vertices)
is given in the paper \cite{DSV09} by Damron et al, in particular Lemma 6.3). 

Since in our case $\delta \asymp 1/R$, the characteristic length is at least of order $R$ (the length of the box). Hence, the mentioned
four-arm event has, under the measure $\P$ roughly the same probability as under $P_{\frac{1}{2}}$. 
Summing that probability over all vertices in the `bulk' of the box would then also yield Theorem \ref{thm-cross-piv}.

\end{subsubsection}

\begin{subsubsection}{Some remarks on \cite{SS10}} \label{sect-rem-ScSt}
Recall again that Schramm and Steif \cite{SS10} worked on Bernoulli percolation (not on the percolation model with fixed number
of occupied vertices).  We also remark that it was not the goal of
their paper to obtain lower bounds for the expected number of pivotals but to prove a form of quantitative noise sensitivity (and use that to prove the
existence of exceptional times in a dynamical percolation model).
They used, apart from the above mentioned exploration paths, not OSSS but a different
inequality (Theorem 1.8 in their paper), which involves discrete Fourier analysis. In fact, a special case in their
Theorem 1.8 produces an {\it upper} bound for the expected number of pivotals.

\end{subsubsection}

\begin{subsubsection}{Comparison with KKL in this box-crossing situation}
We also briefly return to the $k$-out-of-$n$ version of KKL from \cite{OWi13}, mentioned in Section \ref{rel-work}.
Since each vertex $v$ in the $R \times R$ box has distance at least $R/2 - 1$ to the left or to
the right side of the box, the probability that $v$ is pivotal is at most 
$\P(\exists \text{ occupied path from a neighbour of } v \text{ to distance } \frac{R}{2} - 1 \text{ from } v)$, 
which by Observation \ref{triv-obs} and \eqref{1arm-bnd} is at most $2 c_3 (R/2 - 1)^{-c_4}$. Since this holds for each $v$ in the box,
application of the above mentioned version of KKL then gives that the expected number of pivotal vertices is at least of
order $\log R$, which is much weaker than \eqref{eq-thm4.1}.
\end{subsubsection}

\end{subsection}

\end{section}

\bigskip\noindent
{\bf \large Acknowledgments}

\smallskip
The first author (JvdB) thanks Jeff Kahn and Gabor Pete for some valuable comments and questions on the first version of this paper. In particular, JK brought the
paper \cite{OWi13}
by O'Donnell and  Wimmer to JvdB's attention, and GP asked stimulating questions about the logarithmic factor in the main result in that earlier version.

Further, JvdB thanks Pierre Nolin for helpful information concerning near-critical Bernoulli percolation. 

\end{document}